\documentclass[a4paper,leqno,12pt,twoside]{amsart}
\usepackage{amsmath,amsfonts,amssymb,amsthm,amscd}
\usepackage[utf8]{inputenc}
\usepackage[T1]{fontenc}
\usepackage[english]{babel}
\usepackage[top=1.2in,bottom=1.2in,left=1.in,right=1.in]{geometry} 
\usepackage{graphicx}
\usepackage{mathtools}
\usepackage{paralist}
\usepackage[usenames,dvipsnames,svgnames,table]{xcolor}
\usepackage{tabto}
\usepackage{standalone}
\usepackage{tikz}
\usetikzlibrary{matrix}
\usetikzlibrary{arrows}
\usepackage{xfrac}
\usepackage{amsthm, amsmath, amssymb,latexsym}
\usepackage{tikz-cd}
\usepackage[all]{xy} 

\usepackage{enumerate}
\usepackage{xcolor}
\usepackage{aliascnt}
\usepackage{mathabx}

\usepackage[colorlinks=true,citecolor=blue, urlcolor=blue, linkcolor=blue]{hyperref}
\usepackage{cleveref}

\usepackage{textcomp}

\def\ra{\rightarrow}

\newcommand{\tsk}[1]{\textcolor{YellowOrange}}

\newtheorem{teo}{Theorem}[section]

\newtheorem*{mainteoa}{Theorem A}
\newtheorem*{mainteob}{Theorem B}
\newtheorem{defin}[teo]{Definition}
\newtheorem{prop}[teo]{Proposition}
\newtheorem{cor}[teo]{Corollary}

\newtheorem{lemma}[teo]{Lemma}

\theoremstyle{definition}

\theoremstyle{definition}
\newtheorem*{ack}{Acknowledgement}
\newtheorem{remark}[teo]{Remark}

\newtheorem{question}[teo]{Question}
\newtheorem*{question*}{Question}

\newcommand{\Pic}{\operatorname{Pic}}

\title[Remarks on the positivity of the cotangent bundle of an Enriques surface]{Remarks on the positivity of the cotangent bundle of an Enriques surface}

\author{Dario Faro}

\address{Dario Faro  \\ Universit\`a degli Studi di Milano  \\ Dipartimento di Matematica \\ Via Saldini 50  \\ 20133 Milano, Italy  }
 \email{dario.faro@unimi.it}

\begin{document}

\begin{abstract}
Let $S$ be a very general Enriques surface. In this paper we show that $\Omega_{S}|_C$ is semistable for a general curve $C \in |H|$, where $H$ is an effective line bundle on $S$. Moreover, we give explicit constructions of families of smooth irreducible curves that destabilize $\Omega_S$.
\end{abstract}

\thanks{
	The author is a member of GNSAGA (INdAM), and was   partially supported by PRIN project {\em Moduli spaces and special varieties} (2022).  The author  is also grateful to the Max Planck Institute for Mathematics in Bonn for its hospitality and financial support while part of this work was carried out.
}

\maketitle

\section{Introduction} Let $S$ be a smooth projective variety over an algebraically closed field $k$ of characteristic zero and let $H$ be an ample line bundle on $S$. We recall that if $E$ is a torsion-free sheaf its slope is defined as \[ \mu_H(E):=\frac{\operatorname{deg}_H(E)}{\operatorname{rank}(E)}=\frac{c_1(E) \cdot H^{\operatorname{dim}(S)-1}}{\operatorname{rank}(E)}, 
\]
and $E$ is said to be $\mu$-(semi)stable if $\mu(F)< \mu(E)$ ($\leq $) for all subsheaves $F \subset E$ with $0 < \operatorname{rank}(F) < \operatorname{rank}(E)$. In this note we consider the situation when $S$ is an Enriques surface and $E=\Omega_S$ is the cotangent bundle. In this situation $\Omega_S$ is stable because the pullback to its $\operatorname{K3}$ cover is.  

In this note, we are interested in studying the $\mu-$semistability of the restriction of the cotangent bundle of $S$ to general smooth curves in the linear system $|H|$. The problem of studying the semistability of the restriction of vector bundles to general hypersurfaces is classical. We mention a few important results. The theorems of Mehta and Ramanathan assert that for a smooth polarized variety $(S,H)$ the restriction of a $\mu-$(semi)stable sheaf $E$ is (semi)stable for general hypersurfaces in $|mH|$ when $m >>0$ (\cite{MR1}, \cite{MR2}), while the restriction theorem of Flenner (\cite{F}) gives effective bounds on $m$ depending on the rank of $E$, the degree of $H$, and the dimension of $S$. Applied to the case of surfaces,  it yields the semistability for the restriction to the general element when $m \geq 2H^2$. A result of Hein (\cite{H1}), which is valid in any characteristic, gives instead effective bounds for the semistability of the restriction of rank $2$ semistable vector bundles, depending only on $H^2$ and the discriminant. When $S$ is an Enriques surface and $E=\Omega_S$, this gives that $\Omega_{S}|_{C}$ is semistable for general $C \in |2^lH|$ whenever $l \geq \log_2(\sqrt{\frac{48}{H^2}+1})$. Another important result for surfaces is Bogomolov's theorem. If $E$ is a locally free $\mu-$stable sheaf on $S$, the theorem says that $E_{|C}$ is $\mu$-stable for \emph{every} smooth curve in $|mH|$, when $m$ is sufficiently positive. The bound on $m$ depends only on some invariants of $E$ (\cite[Theorem 7.3.5]{HL}). When $S$ is an Enriques surface, and $E=\Omega_S$ it yields $m \geq 48$. A similar result, valid in any characteristic,  is proved by Hein (\cite[Theorem 2.8]{H1}). In the case of Enriques surfaces this says that $\Omega_{S}|_C$ is semistable for any $C \in |mH|$ when $m^2H^2 \geq 48$. In this case $H$ is required to be very ample.  The first main result of this paper is an improvement of the bound for the semistability of the restriction of $\Omega_S$ when $S$ is a very general Enriques surface. The result is  the following (Theorem \ref{maintheorem1}, see also Remark \ref{remarkcomparison2} for comparison).
\begin{mainteoa}
\label{teoa}
Let $S$ be a very general Enriques surface, and let $H$ be an effective line bundle on $S$. Let $C \in |H|$ be a general smooth curve. Then $\Omega_S|_{C}$ is semistable.
\end{mainteoa}
More precisely, we will need $S$ to satisfy Theorem~\ref{teoknutsengalati} of Knutsen and Galati (\cite[Theorem 1.1]{GK}), which shows that for a very general Enriques surface every irreducible rational curve is $2$-divisible in $\operatorname{Num}(S)$.

The proof of the result is based on the following observation. In the very general setting of Theorem~\ref{teoknutsengalati}, every effective line bundle $H$ is linearly equivalent to a sum of the type
\[
a_1E_1+\cdots+a_kE_k+\epsilon E_1',
\]
where the $E_i$'s (and $E_1'$) are smooth elliptic curves (Remark~\ref{remarksinguell}), and $a_i,\epsilon$ are integers. The proof proceeds by induction on the length $n:=a_1+...+a_k+\epsilon$,  and  the type of the decomposition. More precisely, starting from the semistability of $\Omega_S|_E$, when $E$ is a smooth elliptic curve (Lemma~\ref{ellipticfibration}), we construct, inductively on $n$, a nodal reducible curve in $|H|$ having the desired property. We then conclude by openness of semistability. In this argument, a fundamental role is played by the ``adjusted slope'' introduced in \cite{CLV}, which allows us to control semistability on reducible curves in a  simple way (Lemma~\ref{adjustednodalcurvesmultiples}).

In order to introduce the second main result of this paper, we recall that  by a theorem of Hartshorne, a vector bundle of degree zero on a smooth curve is semistable if and only if it is nef. Hence, Theorem A implies that the restriction to general smooth curves in $|H|$ is nef. On the other hand, the cotangent bundle of an Enriques surface is not pseudo-effective. This is an immediate consequence of a theorem of Nakayama (\cite{Nakayama}),  which establishes that the cotangent bundle of $K3$ surfaces is not pseudo-effective. 
 
 The failure of pseudo-effectiveness,  together with a fundamental theorem of Boucksom, Demailly, P\v{a}un and Peternell (\cite{BDPP}), implies that there exist families of curves dominating $S$ such that $\Omega_{S}|_C$ is not nef.  
 The second main result of this note gives an explicit construction of such a family. The example we provide has the additional property that the general member is smooth (see Theorem  \ref{maintheorem2}). 
\begin{mainteob}
Let $S$ be a very  general Enriques surface. Then there exists a line bundle $H'$ with $H'^2=48$ and $\phi(H')=4$, and a positive dimensional family of smooth irreducible curves in $|H'|$ such that $\Omega_{S}|_C$ is not semistable.
\end{mainteob}
Since the restriction of the cotangent bundle of $S$ to a general curve is semistable, we remark that Theorem B provides an explicit construction of a very special, and possibly interesting (for other applications),  family of smooth curves on $S$. 

 The idea of the construction is based on an analogous one for very general $K3$ surfaces of Gounelas and Ottem (\cite{GO}): the family is obtained as a complete intersection of two smooth irreducible divisors in $\mathbb{P}(\Omega_S)$,  and the destabilizing line bundle is the restriction of the tautological line bundle $\mathcal{O}_{\mathbb{P}(\Omega_S)}(1)$.  Given the numerics involved in the Enriques case, however, trying to carry out a similar construction is quite delicate. This is explained in Remark \ref{remarkcomparison}. A fundamental ingredient for our approach in the Enriques case  is the  construction, due to Höring and Peternell (\cite{HP}), of  a very interesting  surface in $\mathbb{P}(\Omega_S)$. The surface is constructed starting with an elliptic fibration on $S$   and, in this case, provides a major obstruction to the nefness of $\Omega_S$.

\subsection{Organization of the paper}
The paper is organized as follows. In Section~\ref{section2}, we first recall some basic facts about isotropic decompositions of effective line bundles and the geometry of half-fibres. We then briefly review the theory of Coskun–Larson–Vogt and give some simple applications to configurations of elliptic curves, and more generally to reducible nodal curves for which the semistability of the restriction of $\Omega_S$ to each irreducible component is known. We also give a separate argument, based on Gieseker semistability, to treat the case of two smooth elliptic curves tangent with intersection multiplicity two. After dealing with a few additional low-length configurations, we combine these ingredients  to prove Theorem~\ref{maintheorem1}.

 In Section \ref{section4} we briefly review some background on the positivity of $\Omega_S$ when $S$ is an Enriques surface. Finally, in Section \ref{section5}, we present an explicit construction of destabilizing curves: this is Theorem \ref{maintheorem2}.

\begin{ack}
I am grateful to Edoardo Sernesi for pointing out the very interesting article \cite{GO} and for explaining the constructions given in that paper; to Frank Gounelas and John Ottem for very interesting conversations on this topic; and to Frank Gounelas, Andreas Höring, and Andreas Knutsen for their very helpful and useful comments and feedback on this note. Finally, I am grateful to Andreas Höring for pointing out his paper \cite{HP} to me and for suggesting that the surface considered in Section~\ref{section5} could provide a good source of negativity for $\Omega_S$.
\end{ack}

\subsection{Convention}
The main results are proved for smooth projective (sometimes very general) Enriques surfaces over an algebraically closed field $k$ of characteristic zero.  If  $E$ is a vector bundle on a variety $X$ we use Grothendieck convention to denote denote $\mathbb{P}(E)$.

\section{Semistability of the restricted cotangent bundle}
\label{section2}
\medskip
\noindent\textbf{Bacground and strategy of the proof.}
\medskip

Let  $S$ be a very general  Enriques surface defined over $k=\bar k$ of characteristic zero. In this section we prove a result for the semistability of the restriction of $\Omega_S$  to the general curve in $|H|$: this is Theorem  \ref{maintheorem1}.   More precisely we will need $S$ to satisfy:
\begin{teo}[\cite{GK}[Theorem 1.1]
\label{teoknutsengalati}
Let $S$ be a very general Enriques surface. If $C \subset S$ is an irreducible rational curve, then $C$ is $2-$divisible in $\operatorname{Num}(S)$. 
\end{teo}
In the statement, by “very general”, we mean that there is a set that is the complement of a countable union of proper Zariski-closed subsets in the moduli space satisfying the given conditions.

The proof of Theorem \ref{maintheorem1} is based on specialization to nodal curves in the linear system $|H|$. Before explaining this, we need to introduce some little background on curves on Enriques surfaces. 

The first main ingredient is the simple  isotropic decomposition of effective line bundles on Enriques surfaces. We borrow from \cite[Section 2]{cdgk2} what follows.

Let $S$ be an Enriques surface. Any effective line bundle on $S$ with $H^2 \geq 0$ is linearly equivalent to:
\[
H \sim a_1E_1+...+a_nE_n+\epsilon K_S
\]
where all $E_i$ are effective, primitive and isotropic, all $a_i$ are positive integers, $n \leq 10$, and $\epsilon \in \{0,1\}$.
Moreover,
\begin{equation}
\label{intersectionelliptic}
\begin{cases}
\text{either } n\neq 9,\quad E_i\cdot E_j=1 \text{ for all } i\neq j,\\
\text{or } n\neq 10,\quad E_1\cdot E_2=2
\text{ and otherwise } E_i\cdot E_j=1 \text{ for all } i\neq j,\\
\text{or } E_1\cdot E_2=E_1\cdot E_3=2
\text{ and otherwise } E_i\cdot E_j=1 \text{ for all } i\neq j.
\end{cases}
\end{equation}

The following remark is crucial for the rest of this section.

\begin{remark}
\label{remarksinguell}
Any effective primitive isotropic line bundle $E$ is also called a half-fiber. Indeed, the linear system $|2E|$ defines a genus-one fibration on $S$, whose only multiple fibers are two double fibers $2E_1$ and $2E_2$. One has $h^0(E_i)=1$, and up to reordering:
\[
E_1 \sim E, \qquad E_2 \sim E+K_S,
\]
(see \cite[Proposition 2.2.8 and Corollary 2.2.9]{CDL}). We recall that an Enriques surface is called unnodal if it does not contain any smooth rational curves. The general Enriques surface (in the moduli space) is unnodal. If $S$ is very general (in the sense of Theorem \ref{teoknutsengalati}), then it is easy to see that $S$ is is  unnodal (using that $R^2=-2$ for a smooth rational curve). If  $E$  is an half-fiber of an elliptic fibration on $S$, then $E$ is either a smooth elliptic curve, a rational curve with a node, or a reducible nodal curve (i.e.\ of type $I_b$, $b \ge 0$, in Kodaira's classification of singular fibres \cite[V.7]{BHPV}). If $S$ is as in Theorem  \ref{teoknutsengalati}, we can then  assume that $E$ is of type $I_0$ (smooth elliptic curve). 
\end{remark}
Thanks to the previous remark we can then assume that 
\begin{equation}
\label{decomposition}
H \sim a_1E_1+\cdots+a_nE_n+  \epsilon E'_1,
\end{equation}
where  where all $E_i$, $i=1,..,n$ are smooth elliptic curves satisfying  \eqref{intersectionelliptic}, $\epsilon \in \{0,1\}$, and $E'_1 \sim E_1 +K_S$ is the other half fiber of the elliptic fibration  associated with $|2E_1|$. 

Now we recall the definition of the $\phi$-function which was introduced by Cossec (for details see for example \cite{CDL}).

\begin{defin}
	\label{phi}
	Let $H$ be a big and nef line bundle on $S$.
    \[
\phi(H):=\text{min}\{|H \cdot F|: F\in \Pic(S), F^2=0, F \not \equiv 0\}.
\]
\end{defin}
The value $\phi(H)$ is always bounded from above: 
\begin{equation}
\label{inequalityHquadro}
H^2 \geq \phi(H)^2.
\end{equation}
The $\phi-$function measures the ``regularity" of the (rational) map associated with a line bundle.  This is well-expressed by the following. 
\begin{teo}\cite[Theorem $2.4.14$, Theorem $2.4.18$ and $2.4.19$]{CDL}
\label{etoegkveraglob}
Let $S$ be an Enriques surface and let $L$ be a big and nef line bundle on $S$. Then 
\begin{itemize}
    \item $|L|$ is base point-free if and only if $\phi(L) \geq 2$.
    \item $|L|$ is very ample if and only if $\phi(L) \geq 3$ and there exists no effective divisor $E$ on $S$ such that $E^2=-2$.
\end{itemize}
\end{teo}

Before stating the main result of the section, we prove the following simple lemma.

\begin{lemma}
\label{remarksmoothcurve}
Let $S$ be a very general Enriques surface in the sense of
Theorem~\ref{teoknutsengalati}, and let $H \not \simeq \mathcal{O}_S$ be an effective line bundle on $S$.
Then the linear system $|H|$ contains a smooth irreducible curve, except in the
cases cases $H \sim mE_1$ with $m \geq 3$, or 
$H \sim mE_1+E_1' $ with $m \geq 1$.  In both cases, a general element of $|H|$ is a disjoint union of smooth
irreducible elliptic curves.
\end{lemma}

\begin{proof}
By the decomposition~\eqref{decomposition}, we may write
\[
H \sim a_1E_1+\cdots+a_nE_n+\epsilon E'_1,
\]
where the $E_i$ are smooth elliptic curves satisfying
\eqref{intersectionelliptic}, $\epsilon\in\{0,1\}$, and
$E'_1\sim E_1+K_S$ is the other half-fibre of the elliptic fibration
associated with $|2E_1|$. Assume first that $\phi(H)\geq 2$. Then, by
\eqref{inequalityHquadro}, one has $H^2\geq \phi(H)^2>0$. Hence, by Bertini and Theorem~\ref{etoegkveraglob}, a general element
$C\in |H|$ is smooth and irreducible. It remains to consider the cases in which either $H^2=0$, or
$\phi(H)=1$.

First suppose that $H^2=0$. Then either $H\sim mE_1$, or $H\sim mE_1 +E_1'$. Consider first the case $H\sim mE_1$.  When 
$m=1,2$, a general element of $|mE_1|$ is a smooth irreducible
elliptic curve, by Remark~\ref{remarksinguell}. On the other hand, if
$m\geq 3$, then the  general element of $|mE_1|$ is reducible. More
precisely, if $m=2m'$, $m' \geq 2$,  then a general element is the disjoint union of
$m'$ (smooth) fibres of the elliptic fibration associated with $|2E_1|$. This follows easily using that  $h^0(S,2m'E_1)=m'+1$ (by \cite{CDL}[Corollary 2.9]), together with the fact that  the dimension of $\sum_{i=1}^{m'} |2E_1|\subset |2m'E_1|$ is $m'$. If
$m=2m'+1$, $m' \geq 1$ then a general element is the disjoint union of the fixed
half-fibre $E_1$, and $m'$ fibres of the  elliptic fibration. This use again \cite{CDL}[Corollary 2.9].

Now suppose that $H\sim mE_1 +E_1'$, $m \geq 1$. If $m=2m'$, then we conclude again applying \cite{CDL}[Corollary 2.9]),  and using that $\sum_{i=1}^{m'} |2E_1|+E_1' \subset |2m'E_1+E_1'|$ has dimension $m'$. If  $m=2m'+1$, then $H\sim mE_1 +E_1' \sim 2m'E_1+E_1+E_1'$. A similar computation as in \cite{CDL}[Corollary 2.9]), shows that  $h^0(S,2m'E_1+E_1+E_1')=m'+1$. Then we conclude observing that the subspace $\sum_{i=1}^{m'}|2E_1| + E_1 + E_1' \subset |H|$ has dimension $m'$.

Now suppose that $\phi(H)=1$. Since $S$ is very general (in particular it does not contain any smooth rational curve), by   \cite{CDL}[Propositions 2.6.1],   $H\sim mE_1+E_2$ with $E_1\cdot E_2=1$.  By  \cite{CDL}[Propositions 2.6.1 and 2.6.4] the base point are simple (that is the multiplicity of the (scheme theoretic) base locus at each point is $1$). This implies that there exists $C \in |H|$ which is smooth at the base points.  Thus a  general $C \in |H|$ is  smooth and irreducible by Bertini.
\end{proof}

Let $S$ be a very general Enriques surface, and let $H \neq \mathcal{O}_S$ be an effective line bundle on $S$. By  Lemma \ref{remarksmoothcurve} we can assume that the general element of $H$ is either a smooth irreducible curve, or  a disjoint union of smooth elliptic curves. The main result of the section is the following:

\begin{teo}
\label{maintheorem1}
Let $S$ be a very general Enriques surface, and let $H$ be an effective line bundle on $S$. Let $C \in |H|$ be a general smooth curve. Then $\Omega_S|_{C}$ is semistable.
\end{teo}

Of course, when a general member of $|H|$ is a disjoint union of smooth
irreducible curves, we mean that the restriction of $\Omega_S$ to each
curve is semistable. The strategy of the proof is quite simple: we proceed by induction on the number of half-pencils in the decomposition \eqref{decomposition}. For each possible length of the decomposition, we construct a nodal curve $C\in |H|$ (or, in one case, a curve with a tacnode) such that $\Omega_S|_C$ is semistable. This will imply that the restriction of $\Omega_S$ to a general member of $|H|$ is semistable.

\medskip
\noindent\textbf{The nodal case}
\medskip

 To deal with the nodal cases, we  will use the following notion of adjusted slope introduced by Coskun, Larson and Vogt (\cite[Definition 2.1]{CLV}). 
\begin{defin}
Let $E$ be a vector bundle on a connected nodal curve $C$. Let $\nu: \tilde C \rightarrow C$ be the normalization. For a subbundle $F \subset \nu^* E$, define the adjusted slope 
\begin{equation}
\label{adjustedslope}
\mu_{C}^{adj}(F):=\mu(F)-\frac{1}{rkF}\sum_{p \in C_{sing}}codim_F(F_{|\tilde p_1} \cap F_{|\tilde p_2}),
\end{equation}
where  $\tilde{p}_1$ and $\tilde{p}_2$ are the preimages of the node $ p \in C$, and $codim_F(F_{|\tilde p_1} \cap F_{|\tilde p_2})$ refers to the codimension of the intersection in either $F_{|\tilde p_1}$ or $F_{|\tilde p_2}$. We say that $E$ is (semi)stable if for all subbundles $F \subset \nu^*E$ 
$$
\mu^{adj}_C(F) < (\leq) \mu^{adj}_C(\nu^*E)(=\mu(\nu^*E)).
$$
\end{defin}
\begin{remark}
\label{opensem}
With this definition we have that semistability is open in families of connected nodal curves  (\cite[Proposition 2.3]{CLV}).
\end{remark}
We start with the following  basic lemma.

\begin{lemma}
\label{adjustednodalcurves}
Let $C'_1$, $C'_2$ be two smooth irreducible curves on an Enriques surface $S$ intersecting transversely at $n \geq 1$ points. Suppose that $\Omega_{S}|_{C'_i}$ is semistable for  $i=1,2$. Set $C:=C'_1 \cup C'_2$. Then $\Omega_{S}|_{C}$ is semistable.
\end{lemma}

\begin{proof}
Let $ \nu: \tilde S \ra S$ be the blowup of $S$ at the $n$ points of intersection of $C'_1$ and $C_2'$, and let $\tilde C \subset \tilde S$
be the strict transform of $C$. Then
$\nu_{|\tilde C}:\tilde C\ra C$ is the normalization map. Observe that $\tilde C=C_1 \sqcup C_2$ where $C_1 \simeq C_1'$ and  $C_2 \simeq C_2'$. Let $ F \subset \nu^*\Omega_{S}|_{\tilde C}$ a sub line bundle. We have: 
\begin{align*}
\mu^{adj}_C(F) \leq   \mu(F|_{C_1}) +  \mu(F|_{C_2})
\leq  \mu(\nu^*\Omega_{S}|_{C_1}) +  \mu(\nu^*\Omega_{S}|_{C_2}) &=\mu(\nu^*\Omega_{S}|_{ \tilde C}) \\
&=\mu^{adj}_C(\nu^*\Omega_{S}|_{ \tilde C})
\end{align*}
Observe that the first inequality follows from \eqref{adjustedslope} and the additivity of  the classical slope.  For the second inequality we have used that $\nu_{|C_i}: C_i \ra C'_i$ is an isomorphism, and the semistability of $\nu^*\Omega_{S}|_{C_i}$. Hence $\Omega_{S}|_C$ is semistable.
\end{proof}

The previous lemma has the  following immediate  generalization:
\begin{lemma}
\label{adjustednodalcurvesmultiples}
Let $C=C_1\cup...\cup C_k \subset S$ be a connected nodal curve whose
irreducible components are smooth. Suppose that $\Omega_S|_{C_i}$
is semistable for every $i$. Then $\Omega_S|_C$ is semistable.
\end{lemma}
\begin{proof}
The proof is identical to that of Lemma~\ref{adjustednodalcurves}, with the sums taken over all irreducible components of $C$.
\end{proof}

Now we specialize to the situation where the $E_i:=C_i'$  of the previous lemma are smooth elliptic curves.  The following is very well-known.

\begin{lemma}
\label{ellipticfibration}
Let $E$ be an elliptic curve on an Enriques surface $S$. Then $\Omega_{S}|_{E}$ is semistable (but not stable).
\end{lemma}
\begin{proof}
Let $L  \xhookrightarrow{}   \Omega_{S}|_{E}$ be a line bundle. Consider the conormal exact sequence:
$$
0 \rightarrow \mathcal{O}_{E}(-E) \rightarrow \Omega_{S}|_{E}\rightarrow \omega_{E}  \rightarrow 0.
$$
 Then either $L \hookrightarrow \mathcal{O}_{E}(-E)$, or   $L \hookrightarrow \omega_{E}$. Since $\operatorname{deg}(\Omega_{S}|_{E})=\operatorname{deg}(\omega_E)=\operatorname{deg}(\mathcal{O}_{E}(-E))=0$, we immediately conclude that $\Omega_{S}|_{E}$ is semistable. 
\end{proof}

\begin{cor}
\label{lemmablowupelliptic0}
Let $S$ be an Enriques surface, and let $C$ be a nodal curve consisting of two smooth irreducible elliptic curves $E_1$ and $E_2$ intersecting transversely in $n \geq 1$ distinct points.  Then $\Omega_{S}|_C$ is semistable.
\end{cor}
\begin{proof}
This follows from Lemma \ref{ellipticfibration}, together with Lemma \ref{adjustednodalcurves}.
\end{proof}

\medskip
\noindent\textbf{Curves with an ordinary tacnode} 
\medskip

There is another basic case we have to deal with, i.e. the case of reducible curves consisting of two irreducible smooth curves tangent at a point with intersection multiplicity $2$. To deal with this we use the notion of semistability (with respect to a polarization) given by the Hilbert polynomial (see \cite{HL}[Definition 1.2.4]). This is also sometimes called Gieseker semistability.  We recall the definition.
\\
\

Let $C$ be a reducible (but reduced) curve. Let $H$ be an ample invertible sheaf on $C$. Let $G$ be a coherent sheaf on $C$ of dimension $1$. Then the Hilbert polynomial of $G$ is
\[
P(G,m)=\chi(G \otimes(mH))=\alpha_0(G)+ \alpha_1(G)m,
\]
where $\alpha_0(G)=\chi(G)$ is the Euler characteristic, and $\alpha_1(G) >0$ is called multiplicity. The reduced Hilbert polynomial is defined as:
\[
p(G):=\frac{P(G,m)}{\alpha_1(G)}.
\]
We recall that a pure, dimension $1$ coherent sheaf on $C$ is semistable if for any  proper subsheaf $G' \subset G$, one has $p(G') \leq p(G)$, or equivalently:
\[
\frac{\chi(G')}{\alpha_1(G')} \leq \frac{\chi(G)}{\alpha_1(G)}.
\]
 If $G$ is a pure coherent sheaf of dimension one on a reduced curve $C$, we will refer to $\frac{\alpha_0(G)}{\alpha_1(G)}$ as the slope of $G$. The following proof is very similar to that of \cite[Lemma 1.2.1]{HL}.

\begin{lemma}
\label{lemmaalpha1}
  Let $G$ be a pure one dimensional sheaf on a reduced curve $C$. Suppose that $C$ has two irreducible smooth components: $C_1$, $C_2$.  Set $h_i:=H \cdot C_i$, and let $r_i$ be the generic rank of $G|_{C_i}$.  Then 
\[
\alpha_1(G)=r_1h_1+r_2h_2.
\]
\end{lemma}
\begin{proof}
For $ m >> 0$ $mH$ is globally generated. By  \cite{HL}[Lemma 1.1.12] the general $D \in |mH|$ is a reduced divisor, and is $G-$ regular (that is $G(-mH) \rightarrow G$ is injective).  In particular we have:
\[
0 \rightarrow G(-mH) \rightarrow G \rightarrow G|_{D} \rightarrow 0.
\]
Then after twisting by $\mathcal{O}_C(mH)$, we get:
\[
\alpha_1(G)=\frac{P(G,m)-P(G,0)}{m}=\frac{\chi(G \otimes mH) - \chi (G)}{m}=\frac{\chi(G\otimes mH|_{D})}{m}=\frac{r_1mh_1+r_2mh_2}{m}
\]
\[
=r_1h_1+r_2h_2
\]
Observe that we have used that for $m >>0$ a general $D \in |mH|$ meets $C_i$ in $mh_i$ distinct points contained in the open subscheme $U_i$ such that $G|_{U_i}$ is locally free of rank $r_i$. 
\end{proof}

\begin{lemma}
\label{ordinarytacnode}
Let $S$ be an Enriques surface, and let $E_1,E_2\subset S$ be two smooth elliptic curves.  Assume  that their scheme-theoretic intersection
\[
Z:=E_1 \cap E_2
\]
is a zero-dimensional subscheme of length $2$ supported at a single point $p \in S$.  Let $H$ be  an ample line bundle on $S$ such that $H \cdot E_1= H \cdot E_2$. Set $E:=E_1 \cup E_2$. Then $\Omega_{S}|_E$ is $H-$ semistable. 
\end{lemma}

\begin{proof}
There is a short exacst sequence in $E$  given as:
\begin{equation}
\label{restrictionell}
0 \rightarrow \mathcal{O}_{E} \rightarrow \mathcal{O}_{E_1} \oplus  \mathcal{O}_{E_2} \rightarrow \mathcal{O}_{Z} \rightarrow 0.
\end{equation}
Let $i: E \hookrightarrow S$ be the inclusion. Tensoring by $i^*\Omega_S$ we get:
\begin{equation}
\label{restrictionell2}
0 \rightarrow \Omega_S|_{E} \rightarrow \Omega_{S}|_{E_1} \oplus  \Omega_S|_{E_2} \rightarrow \Omega_S|_{Z} \rightarrow 0
\end{equation}
Using that $H^0(\Omega_S|_{Z}) \simeq k^{\oplus 4} $, additivity of $\chi$, and Riemann-Roch, we find: $\chi(\Omega_S|_{E})=-4$. Now observe that $\Omega_{S}|_E$ is pure of dimension $1$ on $E$. Indeed if $F \subset \Omega_S|_E$ is a proper subsheaf then \ref{restrictionell2} gives $F \xhookrightarrow{} \Omega_{S}|_{E_1} \oplus \Omega_{S}|_{E_2}$. Then we conclude using that $\Omega_{S}|_{E_i}$ is locally free. Then we can apply \ref{lemmaalpha1}. Together with   the fact that $h_1=H \cdot E_1=H \cdot E_2=h_2$ we get:
\[
\frac{\chi(\Omega_S|_{E})}{\alpha_1(\Omega_S|_E)}=\frac{-4}{4h}=\frac{-1}{h}.
\]
Now let $F \subset \Omega_{S|E}$ be a proper subsheaf. We need to show that 
\begin{equation}
\label{inequalitysemis}
   \frac{\chi(F)}{\alpha_1(F)} \leq \frac{-1}{h}. 
\end{equation}
First we  observe that $F$ is pure dimension $1$ (if not zero).   Indeed this immediately follows from that fact that $F \subset \Omega_{S}|_{E}$, and the latter is locally free. Then by \ref{lemmaalpha1}:
\[
\frac{\chi(F)}{\alpha_1(F)}=\frac{\chi(F)}{h(r_1+r_2)},
\]
where $r_i$ is the generic rank of $F|_{E_i}$. Then \eqref{inequalitysemis} is equivalent to
\[
\chi(F) \leq -(r_1+r_2).
\]
Denote by $F_i:=\operatorname{Im}(F \rightarrow \Omega_{S}|_{E_i})$ coming from the composition of the inclusion $F \rightarrow \Omega_{S}|_{E}$ and the first arrow of \eqref{restrictionell2}. Then $F_i \subset \Omega_{S}|_{E_i}$ is a locally free sheaf of rank $r_i$ on the smooth curve $E_i$. We have a commutative diagram of sheaves on  $S$:
\begin{equation}
\label{diagram4}
\xymatrix{
 & 0 \ar[d]  & 0 \ar[d] & 0 \ar[d] & \\
0 \ar[r] & K \ar[r]  \ar[d]  & F_1 \oplus F_2  \ar[r]  \ar[d]  & Q  \ar[r]   \ar[d]  & 0  \\
0 \ar[r] & \Omega_{S}|_{E}  \ar[r] &  \Omega_{S}|_{E_1} \oplus \Omega_{S}|_{E_2}  \ar[r] & \Omega_S|_{Z} \ar[r] &0 ,
}
\end{equation}
where the bottom row is \eqref{restrictionell2}, $Q:=\operatorname{Im}(F_1 \oplus F_2 \rightarrow \Omega_S|_{Z})$, and $K$ is the kernel of $F_1 \oplus F_2 \rightarrow Q$. The vertical arrows are injective by construction. No consider the commutative diagram:
\[
\xymatrix{
& 0 \ar[d]  & 0 \ar[d] &  \\
0 \ar[r] & F \ar[r] \ar[d] & F_1 \oplus F_2 \ar[d] &  &\\
0 \ar[r] & \Omega_{S}|_{E} \ar[r] & \Omega_{S}|_{E_1} \oplus \Omega_{S}|_{E_2} \ar[r]  & \Omega_S|_{Z} \ar[r] &0
}
\]
Here the injectivity of  the map $F \rightarrow F_1 \oplus F_2$ follows from the commutativity (and the injectivity of the other maps). From the diagram  we have that the composition
\[
F \rightarrow F_1 \oplus F_2 \rightarrow
\Omega_{S}|_{E_1} \oplus \Omega_{S}|_{E_2}
\rightarrow \Omega_S|_{Z}
\]
is zero. Hence using \eqref{diagram4} also the composition $F\rightarrow F_1 \oplus F_2 \rightarrow Q$ is zero. Hence $F \hookrightarrow K$. Now notice that $K/F$ is supported in dimension  zero since on the complement $U$ of $Z$ one easily sees that:
\[
F|_U \simeq (F_1\oplus F_2)|_U \simeq K|_U.
\]
Hence $\chi(K/F) \geq 0$ and $\chi(F) \leq \chi(K)$. Then in order to  bound  $\chi(F)$, we can focus on $\chi(K)$. Using that   $E_i$ are smooth elliptic curves, we get:
\begin{equation}
\label{inequalityquasifinal}
\chi(F) \leq \chi(K)= \chi(F_1) + \chi(F_2) -\chi(Q)=\operatorname{deg}(F_1)+deg(F_2)-l(Q),
\end{equation}
where $l(Q)$ denotes the length of the zero-dimensional sheaf $Q$. For $i=1,2$, let
\[
A_i:=\operatorname{Im}(F_i\rightarrow\Omega_S|_{Z}),
\]
and let $K_i$ be the kernel of the map $F_i\rightarrow A_i$. Then we have a commutative diagram on $E_i$:
\begin{equation}
\label{diagram5}
\xymatrix{
0 \ar[r] & K_i \ar[r] \ar[d] & F_i \ar[r] \ar[d] & A_i \ar[d] \ar[r]  & 0\\
0 \ar[r] & \Omega_{S}|_{E_i}(-Z)\ar[r] & \Omega_{S}|_{E_i}  \ar[r]  & \Omega_S|_{Z} \ar[r] &0 .
}
\end{equation}
Since $\Omega_{S}|_{E_i}$ is semistable, $\Omega_{S}|_{E_i}(-Z)$ is also semistable. Its ($\mu$-) slope is $\frac{-4}{2}=-2$. Thus $\operatorname{deg}(K_i)\leq -2r_i$. By exactness of the first row of \eqref{diagram5}, we get:
\[
\operatorname{deg}(F_i)=\operatorname{deg}(K_i)+l(A_i)
\leq -2r_i+l(A_i).
\]
 Then \eqref{inequalityquasifinal} becomes:
\begin{equation}
\label{finalinequality2}
\chi(F) \leq \chi(K)=\operatorname{deg}(F_1)+deg(F_2)-l(Q)\leq -2r_1-2r_2+l(A_1)+l(A_2)-l(Q).
\end{equation}
Now we observe that $l(A_i) \leq l(F_i|_{Z})=2r_i$. Indeed this follows immediately  from the surjection $F_i|_{Z} \twoheadrightarrow A_i$ coming from the factorization $F_i \rightarrow F_i|_{Z} \rightarrow\Omega_{S}|_Z$.  Moreover, by construction, $A_i \subset Q$, hence $l(A_i)-l(Q) \leq 0$. Then \eqref{finalinequality2} becomes:
\[
\chi(F) \leq  -2r_1-2r_2+l(A_1)+l(A_2)-l(Q) \leq -2r_1-2r_2+l(A_i) \leq -2r_1 -2r_2 + 2r_i
\]
for $i=1,2$. This gives $\chi(F) \leq -2r_i$ for $i=1,2$ and thus implies  $\chi(F) \leq -r_1-r_2$ as desired.
\end{proof}

\begin{cor}
\label{cortacnodesemistability}
Let $E_i \subset S$, $i=1,2$ be two smooth elliptic curves on $S$ (very general Enriques surface) tangent at a point with multiplicity $2$. Set $H:=E_1+E_2$. Then  the restriction $\Omega_{S}|_{C}$ to a general element in $|H|$ is $\mu$-semistable.
\end{cor}

\begin{proof}
First observe that $H$ is an ample line and globally generated line bundle (for example by Theorem \ref{etoegkveraglob}). Moreover $H \cdot E_1=H \cdot E_2=2$. Set $E:=E_1 \cup E_2$. Applying  Lemma \ref{ordinarytacnode} with this choice of $H$, we find that $\Omega_{S}|_{E}$ is Gieseker semistable. By the openness of semistability, we deduce that $\Omega_{S}|_{C}$  is Gieseker semistable for a general smooth irreducible curve $ C \in |H|$. Using \cite[Lemma 1.2.13]{HL} we conclude that $\Omega_{S}|_{C}$ is $\mu$-semistable for general $C \in |H|$. 
\end{proof}

Before coming to the proof of the main Theorem of the section, we  need to cover some other special configuration:

\begin{lemma}
\label{lemma2}
Let $S$ be a very general  Enriques surface. Let $H$ be an effective line bundle with $H \sim E_1 + E_2$, where $E_i$ are smooth elliptic curves, and $E_1 \cdot E_2=1$. Then $\Omega_{S}|_C$ is semistable for a general smooth irreducible curve  $C \in |H|$.
\end{lemma}
\begin{proof}
In the final part of Lemma \ref{remarksmoothcurve} we have observed that a general element in $|H|$ is a smooth and irreducible curve. Set $E:=E_1 \cup E_2$.  By Corollary \ref{lemmablowupelliptic0}, $\Omega_{S}|_{E}$ is semistable.  By the opennes of semistability we conclude.
\end{proof}

\begin{lemma}
\label{lemma5}
Let $S$ be a very general Enriques surface,  and let $H$ be an effective line bundle with $H \sim E_1+E_2+E_3$, and $E_i \cdot E_j=1$ for $i \neq j$.  If $C \in |H|$ is a general smooth irreducible curve, then $\Omega_{S}|_C$ is semistable. 
\end{lemma}

\begin{proof}
One immedately sees that  $\phi(H) \geq 2$. Then a  general element in $|H|$ is a smooth irreducible curve by Bertini and Theorem  \ref{etoegkveraglob}. Denote by $p$ and $q$ the two (simple) base points of $|E_1+ E_2|$.  We already observed that a general $ C' \in |E_1+ E_2|$ is a  smooth and irreducible curve (see the final part of  Lemma \ref{remarksmoothcurve}). Observe that $(E_1+ E_2) \cdot E_3=2$.  Since $p$ ($q$) is a simple base point for the base locus, a general $C' \in |E_1+ E_2|$ cannot be tangent to $E_3$ at $p$ ($q$). Then a general $C' \in  |E_1+ E_2|$ meets $E_3$ in two distinct points. Since $\Omega_{S}|_{C'}$ is semistable by Lemma \ref{lemma2}, we conclude the proof  using Lemma \ref{ellipticfibration} and Lemma  \ref{adjustednodalcurves}. 

\end{proof}

\noindent\textbf{Proof of Theorem \ref{maintheorem1}}

\begin{proof}
Let $H$ be an effective line bundle on a very general Enriques surface. By Lemma \ref{remarksmoothcurve},

\begin{equation} \label{decompositionprooftheorem}
H \sim a_1E_1+\cdots+a_kE_k+\epsilon E'_1,
\end{equation}
where each  $E_i$ ($E_1'$) is a smooth irreducible elliptic curve. Moreover the general element of $|H|$ is a smooth irreducible curve, except in the
 cases $H \sim mE_1$ with $m \geq 3$, or 
$H \sim mE_1+E_1' $ with $m \geq 1$, in  which a general element is a disjoint union of smooth
irreducible elliptic curves. 

The proof is by induction on $n:=a_1+...+a_k+\epsilon$, the number of half-fibres, counted with multiplicity, appearing in the decomposition \eqref{decompositionprooftheorem}.  For each possible value of  $n$ and each decomposition type, we construct a nodal curve $C\in |H|$ (or, in one case, when $n=2$, a curve with a tacnode) such that $\Omega_S|_C$ is semistable. This will imply that the restriction of $\Omega_S$ to a general member of $|H|$ is semistable.

If $H$ is an effective line  admitting a decomposition of length $n=1$, then $H \sim E$, where $E$ is an half pencil and hence a smooth elliptic irreducible curve (in our very general setting, see Remark \ref{remarksinguell}). Then we conclude by Lemma  \ref{ellipticfibration}.

 Now suppose that $n=2$. Then either $ H \sim 2E_1$, $H \sim E_1+E_2$ with $E_1 \cdot E_2=1$, or $H \sim E_1+E_2$ with $E_1 \cdot E_2=2$, or $H \sim E_1+E_1'$, where $ E_1'$ is (as usual) the other half fiber associated with the elliptic fibration corresponding to $|2E_1|$. 
 
 In the first case, a general element $ C \in |H|$ is a smooth irreducible elliptic curve and again we conclude by Lemma \ref{ellipticfibration}. 
 
 If  $H \sim E_1+E_2$ with $E_1 \cdot E_2=1$ then we conclude by Corollary \ref{lemmablowupelliptic0}.
 
 If $H \sim E_1+E_2$ with $E_1 \cdot E_2=2$, we conclude by Corollary  \ref{lemmablowupelliptic0},  or  Corollary \ref{cortacnodesemistability}.

 If $H \sim E_1+E_1'$, then as in the proof  Lemma \ref{remarksmoothcurve},  a general element is the disjoint union of smooth irreducible elliptic curves.  Hence  we conclude again by  Lemma \ref{ellipticfibration}.

 Now consider the case $n=3$. First suppose that  $H \sim H_1 +E$, where $H_1$ is globally generated (with $n=2$), $H_1 \cdot E >0$, and $E$ is a smooth irreducible elliptic curve.   Then we easily conclude. Indeed, by Bertini a general element  $C_1 \in |H_1|$ is a smooth irreducible curve which intersect transversely $E$. By the length $2$ previous cases $\Omega_S|_{C_1}$ is semistable, and hence $\Omega_S|_{C_1 \cup E}$ is semistable by Lemma \ref{adjustednodalcurves}. Now we consider the remaining case. These correspond to  decomposition of types: 
 $ H \sim 3E_1$, $H \sim 2E_1 +E_1'$, $H \sim E_1+E_2 +E_3$ with $E_1 \cdot E_2 = E_1\cdot E_3=E_2 \cdot E_3=1$. If  $ H \sim 3E_1$, or $H \sim 2E_1 +E_1'$ then (as shown in Lemma  \ref{remarksmoothcurve}),  a general member is a disjoint union of smooth irreducible elliptic curves and we conclude by Lemma \ref{ellipticfibration}. If  $H \sim E_1+E_2 +E_3$ with $E_1 \cdot E_2 = E_1\cdot E_3=E_2 \cdot E_3=1$, we conclude by  Lemma \ref{lemma5}. 

Now  suppose that we know that for  every $H'$ admitting a decomposition of length $n-1 \geq 3$, the restriction  $\Omega_{S}|_{C'}$ to a general $C' \in |H'|$ is semistable.  Let $H$ be an effective line bundle admitting a decomposition of length $n \geq 4$.  Write
\[
H \sim H' + E.
\]
Here $H'$ is an effective line bundle admitting a decomposition of length  $=n-1 \geq 3$, and $E$ a smooth elliptic curve.  There are different cases to consider. 
\begin{itemize}
    \item $H'$ is globally generated, $H'^2 \neq 0$, and $H' \cdot E >0$. Then a  general $ C' \in |H'|$ is a smooth irreducible curve which meets $E$ in $ H' \cdot E$ distinct points. Then we conclude using the inductive hypothesis, Lemma  \ref{adjustednodalcurves} and Lemma \ref{ellipticfibration}.

    \item $H'$ is globally generated, $H'^2=0$, and $H' \cdot E >0$. It is immediate to see that in this case $H' \sim 2m'E'$, where $E'$ is an half pencil. A general element of $|2m'E'|$ is the disjoint union of $m'$  elliptic curves. Moreover we can suppose that each of them meets $E$ transversely in $E \cdot 2E'$ points. Hence  we can find a nodal curve $C \in |H|$ given by the union of $E$ together with these $m'$ elliptic curves. Then we  conclude by applying Lemma \ref{ellipticfibration}, and  Lemma \ref{adjustednodalcurvesmultiples}.

    \item    $H' \cdot E =0$. Using (for example) \eqref{decompositionprooftheorem} this forces $H' \sim a'E + b'E'$, where as usual $E'$ denotes the other half pencil of the elliptic fibration $|2E|$, and $b' \in \{0,1\}$. Hence  $H \sim aE + bE'$ for some $a,b$, $a+b=n$, $b \in \{0,1\}$. As already shown in Lemma \ref{remarksmoothcurve}, the general element is the union of smooth irreducible curves. Hence we conclude.
    
    \item $H'$ is not globally generated, $H' \cdot E > 0$,  $H'^2=0$. Then $H' \sim a'E_1 + b'E_1'$ with $a'+b'=n-1$, $a' \geq n-2 \geq  2,$ $b' \in \{0,1\}$.  Set $E_2:=E$. Then
    \[
    H \sim H' +E_2 \sim a'E_1 + b'E_1' +E_2, 
    \]
    $E_1 \cdot E_2=E_1'\cdot E_2 >0$. After possibly using $E_1' \sim E_1+K_S$, and $E_2' \sim E_2+K_S$,  this can be rewritten as:
    \[
    H \sim aE_1 + E_2 
    \]
    where $a=n-1 \geq 3$, $E_1 \cdot E_2 >0$. Now there are two possibilities, either $E_1 \cdot E_2=1$,  $E_1 \cdot E_2=2$. Consider the first case. Since $a\geq3$, a general member of $|aE_1|$ is a disjoint union of smooth irreducible elliptic curves by Lemma~\ref{remarksmoothcurve}. Moreover, as above, we can choose a connected nodal curve $C\in|H|=|aE_1+E_2|$ consisting of $E_2$ together with these elliptic curves, such that $E_2$ meets each of the other components transversely at distinct points. Then we conclude that the restriction of $\Omega_{S}|_C$ is semistable by  applying Lemma \ref{ellipticfibration}  and Lemma \ref{adjustednodalcurvesmultiples}. 
    
    The other possibility is $E_1 \cdot E_2=2$. Write $H \sim (a-1)E_1 +E_2 +E_1$. Since $(a-1)E_1 +E_2$ admits a decomposition of length $n-1$, the inductive hypothesis gives that the restriction to a general element in $|(a-1)E_1 +E_2|$ is semistable. Notice that  $(a-1)E_1 +E_2$ is globally generated (since $E_1 \cdot E_2 =2$ implies $ \phi((a-1)E_1 +E_2) \geq 2$).  Then a general element $C' \in |(a-1)E_1 +E_2|$ is a smooth irreducible curve meeting $E_1$ in $2$ distinct points and such that $\Omega_{S}|_{C'}$ is semistable. Therefore, as above, $C:=C'\cup E_1\in|H|$
is a nodal curve such that $\Omega_S|_{C}$ is semistable. 
   
 \item $H'$ is not globally generated, $H'\cdot E>0$, and $H'^2>0$. Then $\phi(H')=1$, and 
\[
H'\sim aE_1+E_2,
\qquad
E_1\cdot E_2=1.
\]
(see as usual the  final part of the proof of Lemma  \ref{remarksmoothcurve}). If $E\equiv E_1$, after possibly replacing $E_2$ by $E_2' \sim E_2+K_S$, we can rewrite $H$ as
\[
H\sim(a+1)E_1+\tilde E_2,
\]
and the conclusion follows from the previous case. If $E\not\equiv E_1$, and $ a \geq 3$, set $D:=(a-1)E_1+E_2+E$. Then  $\phi(D) \geq 2$, and  hence it is globally generated. Then the conclusion  follows from the preceding case applied to $H\sim D+E_1$. If $a=2$ and $E\equiv E_2$, we can write $H \sim 2E_1 +2E_2$, or $2E_1 +E_2+E_2'$. In the first case we choose a general $C_1' \in |2E_1|$, and  a general $C_2' \in |2E_2|$ such that they intersect transversely. In the second case we choose a general $C' \in |2E_1|$ which meets both $E_2$ and $E_2'$ transversely. Hence, in both cases one can construct a nodal connected curve whose irreducible components are smooth elliptic curves, and we conclude as above.   Finally, if  $a=2$ and $E \not \equiv E_1$, $E \not \equiv E_2$,  write $H \sim E_1 + E_2 + E+E_1$. Then $\phi(E_1+E_2+E) \geq 2$, and a general element in $|E_1+E_2+E|$ is a smooth irreducible curve which intersects $E_1$ transversely. We then conclude  using the length $n=3$ case and the usual argument.
\end{itemize}
\end{proof}

\begin{remark}
\label{remarkcomparison2}
As alrady noted in the introduction, \cite[Theorem 2.8]{H1} establishes the semistability of $\Omega_{S}|_C$ for a  general $C \in |mH|$, provided that $H$ is a \emph{very ample} line bundle, and $m^2H^2 \geq 48$. For $m=1$, this yelds $H^2 \geq 48$. Hence Theorem    \ref{maintheorem1} covers the remaining range at least when $S$ is a very general Enriques surface (notice that we have no assumption the positivity of $H$, however).  The analogous result for \emph{arbitrary} Enriques surfaces remains  open when $H^2 < 48$.  At a first attempt to this problem, the author tried to adapt the proof of \cite[Theorem 2.1]{DH} in the case of K3 surfaces,  which estabilishes the semistability for general $C \in |H|$  whenever $H$ is globally generated and ample.  However, given the different numerical estimates  involved in the case of Enriques surfaces, one can see that the argument does not extend directly to this setting.
\end{remark}
This  naturally leads to the following question:
\begin{question}
Is it true that, for every Enriques surface $S$ and every ample and globally generated line bundle $H$ on $S$, the restriction $\Omega_S|_C$ is semistable for a general curve $C\in |H|$?
\end{question}

\section{Some remarks about the posivitiy of $\Omega_S$}
\label{section4}
\
In the previous sections we have shown that for a polarized Enriques surface $S$ of degree at least $6$, the restriction $\Omega_{S}|_C$ is semistable for a general curve $C \in |H|$, and the same is true when the degree is $4$, or $2$, and $S$ is very general. On the other hand, it is expected that there exist positive dimensional families  of smooth  curves  $\{C_t\}$ covering $S$ such that $\Omega_{S}|_{C_t}$ is not semistable. As we now explain, this is a consequence of the fact that $\Omega_S$ is not pseudo-effective. 

Let $X$ be a smooth projective variety. By definition, a vector bundle $E$ on $X$ is pseudo-effective if $\mathcal{O}_{\mathbb{P}(E)(1)}$ is pseudo-effective on $\mathbb{P}(E)$.  As usual we denote by $\widebar{\operatorname{Eff}}(X) \subset N^1(X)_{\mathbb{R}}$ the cone of pseudo-effective divisors.  Now we recall the definition of movable curves. We follow \cite{Lazarsfeld}. Denote by $N_1(X)_{\mathbb{R}}$ the real vector space of numerical one cycles. 

\begin{defin}
Let $X$ be an irreducible projective variety of dimension $n$. A class $\gamma \in  N_1(X)_{\mathbb{R}}$ is movable  if there exists a projective birational mapping $\mu: X' \rightarrow X$, together with ample classes $a_1,...,a_{n-1} \in N^1(X')_{\mathbb{R}}$ such that  $\gamma= \mu_{*}(a_1 \cdot...\cdot a_{n-1})$.  The movable cone  $\widebar{Mov}_1(X) \subset N_1(X)_{\mathbb{R}}$   of $X$ is the closed convex cone spanned by all movable classes.
\end{defin}
The link between movable classes and pseudo-effective divisors is given by the following deep theorem  proved in \cite{BDPP}. We state it in the form presented in \cite[Theorem 11.4.19]{Lazarsfeld}.
\begin{teo}
\label{teobp}
Let $X$ be an irreducible projective variety of dimension $n$. Then the cones 
\[
\widebar{\operatorname{Mov}}_1(X) \ \text{and} \ \widebar{\operatorname{Eff}}(X)
\]
are dual.
\end{teo}
Recall that this means that 
\[
\widebar{\operatorname{Eff}}(X)=\{\delta \in N^1(X)_{\mathbb{R}}: \delta \cdot \gamma \geq 0 \ \text{for every} \ \gamma \in \widebar{Mov}_1(X)\}.
\]
As an immediate consequence of \cite{Nakayama} we have the following.
\begin{lemma}
\label{pseudoeffecctive}
Let $S$ be an Enriques surface. Then $\Omega_S$ is not pseudo-effective.
\end{lemma}
\begin{proof}
Let $f:S' \rightarrow S$ be the $K3$ cover. There is an induced \'etale morphism of degree $2$,
$f': \mathbb{P}(\Omega_S') \rightarrow  \mathbb{P}(\Omega_S)$, such that
$f'^{*}\mathcal{O}_{\mathbb{P}(\Omega_S)}(1)=\mathcal{O}_{\mathbb{P}(\Omega_{S'})}(1)$.
By (\cite{Nakayama}), $\Omega_{S'}$ is not pseudo-effective. Hence we conclude using that the
pull-back of pseudo-effective line bundles is pseudo-effective (see for example \cite[Lemma3.4]{HP}).
\end{proof}

Putting together Theorem~\ref{teobp} and Lemma~\ref{pseudoeffecctive}, we obtain that there exists
a positive-dimensional family of curves $\{C_t\}$ on $S$ such that the restriction
$\Omega_{S}|_{C_t}$ is not nef for general $t$. This is standard (see for instance \cite{BDPP}).
For the reader's convenience, we spell out an argument in the following corollary.
\begin{cor}
\label{moving}
Let $S$ be an Enriques surface. Then there exists a positive dimensional family of curves $\{C_t\}$ such that the generic element is irreducible, and $\Omega_{S}|_{C_t}$ is not nef. 
\end{cor}
\begin{proof}
Since $\Omega_S$ is not pseudo-effective, the divisor class
$\mathcal{O}_{\mathbb{P}(\Omega_S)}(1)$ is not pseudo-effective on $\mathbb{P}(\Omega_S)$.
By Theorem \ref{teobp}, there exists a movable class
$\delta\in \widebar{\operatorname{Eff}}(\mathbb{P}(\Omega_S))$ such that
\[
\delta\cdot \mathcal{O}_{\mathbb{P}(\Omega_S)}(1)<0.
\]
By definition of movable class, there exist a projective birational morphism
$\mu:X'\to \mathbb{P}(\Omega_S)$, and ample classes $a_1=\sum a_{1,i} H_i,\ a_2=\sum a_{2,j} H'_j\in N^1(X)_{\mathbb{R}}$, $a_{1,i}$, $ a_{2,j} >0$,   such that $
\delta=\mu_*(a_1\cdot a_2) $.
Hence there exists $H_i$ and $H_j$ such that $\mu_{*}(H_iH_j) \cdot  \mathcal{O}_{\mathbb{P}(\Omega_S)}(1) <0 $. 
Up to replacing $H_i$ and $H_j$ with some multiples,  we can suppose that they are  very ample classes. Hence a general element in $|H_1 \cap H_2|$ is an integral curve $C''_t$  which moves in a positive-dimensional family. Set $C'_t:=\mu(C''_t)\subset \mathbb{P}(\Omega_S)$. Then
\[
\operatorname{deg}(\mathcal{O}_{\mathbb{P}(\Omega_S)}(1)|_{C'_t})<0.
\]
 Denote by $\pi:\mathbb{P}(\Omega_S) \rightarrow S$ the bundle morphism. Now consider the tautological quotient:
\[
\pi^*\Omega_{S}|_{C'_t} \rightarrow \mathcal{O}_{\mathbb{P}(\Omega_S)}(1)_{|C'_t}\rightarrow 0.
\]
This gives:
\[
C'_t \simeq \mathbb{P}(\mathcal{O}_{\mathbb{P}(\Omega_S)}(1)_{|C'_t}) \xhookrightarrow{} \mathbb{P}(\pi^*\Omega_S|_{C'_t})
\]
and $\operatorname{deg}(\mathcal{O}_{\mathbb{P}(\pi^*\Omega_{S}|_{C'_t})}(1)_{|C'_t}) <0$. Set $C_t:=\pi(C'_t)$. Then we conclude that $\Omega_{S}|_{C_t}$ is not nef since its pull-back by the morphism $\pi_{|C'_t}: C'_t \rightarrow C_t$  is not.
\end{proof}

As a consequence of Corollary \ref{moving}, one is then lead to guess the existence of positive dimensional families of smooth irreducible  curves on Enriques surfaces such that $\Omega_{S}|_{C_t}$ is not nef. Equivalently, by Hartshorne's Theorem (\cite[Theorem 2.4]{Ha}), it is equivalent to ask that $\Omega_{S}|_{C_t}$ is not semistable. The aim of the next  section is to produce some (non-trivial) explicit examples of positive dimensional families of  curves satisfying these properties.

\section{Some examples of smooth destabilizing curves}
\label{section5}

A trivial class of examples of destabilizing curves on Enriques surfaces, i.e. curves for which the restriction of $\Omega_S$ is not semistable, is given by smooth rational curves.
.
\begin{lemma}
Let $C \subset S$ be a smooth rational curve. Then $\Omega_{S}|_C$ is not semistable.
\end{lemma}
\begin{proof}
It immediately follows from the conormal exact sequence:
\[
0 \rightarrow \mathcal{O}_S(-C) \rightarrow \Omega_{S}|_C \rightarrow \omega_C \rightarrow 0.
\]
\end{proof}
These, however, are rigid curves, and moreover the general Enriques surface does not contain any smooth rational curve. Another natural class of curves is represented by elliptic curves (recall that every Enriques surface admits an elliptic fibration). For elliptic curves, the restriction of the cotangent bundle is strictly semistable (that is, semistable but not stable), as we have seen in Lemma \ref{ellipticfibration}.

The construction of explicit (non-trivial) examples of positive dimensional families of curves destabilizing the cotangent bundle of (possibly general) Enriques surfaces seems a difficult task. Indeed recall that from Theorem \ref{maintheorem1} the restriction $\Omega_{S}|_C$ is semistable when $C \in |H|$, $H$ is ample and globally generated,  and $C$ is general. Moreover, Bogomolov's theorem(\cite[Theorem 7.3.5]{HL}),  applied to this situation,  gives that  $\Omega_{S}|_C$  is stable for every smooth curve in $|mH|$ when $m \geq 48$.

The analogous problem for $K3$ surfaces  is  studied in \cite{GO}. One of the main results of loc.cit. article is the following:
\begin{teo}{\cite[Theorem A]{GO}}
\label{thmgo}

Let $(S',\mathcal O_S(1))$ be a general polarised $K3$ surface of degree $d$.
\begin{enumerate}
\item[(i)] If $d=2$, then there is a positive-dimensional family of smooth curves
$C \in |\mathcal O_{S'}(6)|$ such that $\Omega^1_{S'|C}$ is not  semistable. 
\item[(ii)] If $d=4$ or $d=6$, there are  families as above in $|\mathcal O_{S'}(3)|$(and another in $|\mathcal{O}_{S'}(4)|$ if $d=4$).
\item[(iii)] If $d=8$, then there exists a positive-dimensional family of smooth curves
$C \in |\mathcal O_{S'}(3)|$ such that $\Omega^1_{S'|C}$ is strictly semistable.
\end{enumerate}
\end{teo}
The contructions in \cite{GO} are fundamental for this work. However it is not so straightforward   to pass from the $K3$ case to the Enriques, even if the  constructions given in \cite{GO} can be generalized to arbitrary polarized $\operatorname{K3}$ (of the same degrees as in the statement). Indeed, given an Enriques surface $(S,H)$, one would be tempted to consider its $K3$ cover $f: S'\rightarrow S$,  and consider the  images  in $S$ of destabilizing curves in $|f^*H|$. However, there is a problem with this approach: the locus of curves $C \in |f^*H|$ such that $f(C)$ is  is a smooth curve is a locally  closed subspace of  $|f^*H|$ (this the locus of smooth curves invariant by the involution), and it does have dimension close to half of the dimension of $|f^*H|$, while the dimension of the families contructed in \cite{GO} is too small to conclude these loci have non empty-intersection. 
\\
\\
The other approach is trying to adapt their construction in our setting. In the following we will  briefly describe the constructions, pointing out how sometimes adapting the constructions in \cite{GO} works, and when we need some more work to do. The first example we want to mention holds for an unnodal Enriques surface. Recall that an unnodal Enriques surface is one that does not contain any smooth rational curve, and the general Enriques is unnodal (see  Remark \ref{remarksinguell}). The structure of the proof is based on the given in \cite[Section 4.1]{GO}, where it is shown that the ramification curve of a double cover $S' \rightarrow \mathbb{P}^2$ destabilizes the cotangent bundle. Here $S'$ denote a $K3$ surface. 
\begin{prop}
\label{doublecover}
Let $S$ be an unnodal Enriques surface, and let $|2F_i|$, $i=1,2$ be two half elliptic pencils on $S$ such that $F_1 \cdot F_2=1$.  Let $f: S' \rightarrow S$ be the $\operatorname{K3}$-cover. Then  there exists a smooth irreducible curve  $R \in |2(f^*F_1 + f^*F_2)|$ such that $C:=f(R)$ is a smooth and irreducible curve of genus $5$, and $\Omega_{S}|_C$ is not semistable.
\end{prop}
\begin{proof}
First observe that if $S$ is an unnodal Enriques surface,  we can always find $F_i$, $i=1,2$, as in the statement (see \cite[Theorem 17.7]{BHPV}). We consider the classical Horikawa construction as described for example in  \cite[pp.345-346]{BHPV}.  Let $F_1$ and $F_2$ be half-pencils with $F_1 \cdot F_2=1$.  The pull-back $f^*F_i$ gives a base-point free linear system and the map associated to $|f^*F_1 + f^*F_2|$ is a finite degree $2$ morphism from $S'$ to a  quadric $Q \simeq \mathbb{P}^1 \times \mathbb{P}^1 \subset \mathbb{P}^3$ branched along a smooth curve $B \subset \mathbb{P}^1 \times \mathbb{P}^1 $ of bidegree $(4,4)$. We denote by $h$ the morphism $S' \rightarrow \mathbb{P}^1 \times \mathbb{P}^1$. Moreover we can fix coordinates $[x_0:x_1:y_0:y_1]$ on $\mathbb{P}^1 \times  \mathbb{P}^1$ such that $B$ is invariant by the involution $i:[x_0:x_1:y_0:y_1] \rightarrow [x_0:-x_1:y_0:-y_1]$. Moreover, we can suppose that the   involution $i$  lifts to the involution  $\tau$ on $S'$ giving the double cover $S' \rightarrow S$. Let $R \simeq B$ be the ramification curve on $X$. Then $\tau(R)=R$, and $R$ is $2:1$ cover of a smooth curve $C:=f(R)$ in $S$ (see \cite{GO}). Since   $\Omega_{S'|R} \simeq  f^*_{|R}\Omega_{S}|_C $ is not semistable, we conclude that $C$ is a destabilizing curve for $\Omega^1_S$.  Since $R \in |h^*\mathcal{O}_{\mathbb{P}^1 \times \mathbb{P}^1}(2,2)|= |2f^*F_1+2F^*_2|$, we conclude that $R^2=16$, $g(R)=9$, $C^2=8$, and $g(C)=5$.
\end{proof}
\begin{remark}
While the previous proposition gives an example of a smooth destabilizing curve in the interesting range covered by Theorem \ref{maintheorem1}, it is still not satisfactory because the curve does not move in a family (of curves satisfying the same property).
\end{remark}
Before giving the main result of this section we briefly recall the contructions in Theorem \ref{thmgo}. Let $(S',H')$ be a polarized $K3$ surface,  denote by  $\pi: \mathbb{P}(\Omega_{S'}) \rightarrow S'$ the bundle morphism, and set  $L:=\mathcal{O}_{\mathbb{P}(\Omega_{S'})}(1)$.  In loc. cit the authors construct some families of curves $\{C'_t\} \subset \mathbb{P}(\Omega_S)$ such that $\pi^*\Omega_{S'}|_{C^{'}_t}$  is not semistable, and $L|_{C^{'}_t}$ is the destabilizing quotient. Then they show that $\pi|_{C^{'}_t}$ is an isomorphism onto the image, which is then a smooth curve $C_t$ such that  $\Omega_{S'}|_{C_t}$ is not semistable. The families given  are of two types:
\begin{enumerate}
    \item $\{C'_t$\} are either a smooth complete intersections of two ample base point-free divisors in $\mathbb{P}(\Omega_{S'})$ of the form $|\pi^*(mH) +L|$. This is the case: $d=2$, $C'_t \in |\mathcal{O}_{S'}(3)|$, and $d=4$, $C'_t \in |\mathcal{O}_{S'}(4)|$ in  Theorem \ref{thmgo}.
    
    \item $\{C'_t\}$ are ramification curves of general projections $S \rightarrow \mathbb{P}^2$. This is the case $d=4,6$,  and $C'_t \in |\mathcal{O}_{S'}(4)|$ in  Theorem \ref{thmgo}. 
\end{enumerate}
One can easily see that trying to carry out the type  $(2)$ construction for an Enriques surface $S \subset \mathbb{P}^{g-1}$ does not work because  of the numerics involved in the  Enriques case. What we do is to follow the second approach. However, the construction does not immediately follow from the $\operatorname{K3}$ case: we strongly rely on the existence of elliptic fibrations on Enriques surfaces, and a construction of Höring and Peternell (\cite{HP}). See also Remark \ref{remarkcomparison}.
\\
\\
From now on let $S$ be an Enriques surface, and $\pi: \mathbb{P}(\Omega_S) \rightarrow S$ the bundle morphism.  Now we will use a construction given in \cite[Section 5.1]{HP}.  

Let $f: S \rightarrow \mathbb{P}^1$  an elliptic fibration. Set $D:= \sum_{p \in \mathbb{P}^1} f^*p-(f^*p)_{red}$. Then the short exact sequence
\[
0 \rightarrow f^* \omega_{\mathbb{P}^1} \rightarrow \Omega_S \rightarrow \Omega_{S/\mathbb{P}^1} \rightarrow 0
\]
induces an exact sequence
\begin{equation}
\label{shortexactsequence}
0 \rightarrow f^* \omega_{\mathbb{P}^1}(D) \rightarrow \Omega_S \rightarrow I_Z \otimes \omega_{S/\mathbb{P}^1}(-D) \rightarrow 0.
 \end{equation}
where $Z$ is a local complete intersection scheme of codimension $2$ with support equal to the singular points of the reduction of the fibres $(f^*p)_{red}$. The  sequence \eqref{shortexactsequence} holds  for any elliptic fibration $f: S \rightarrow C$ from a surface to a curve. Now we briefly describe  a construction.  Consider the short exact sequence:
\begin{equation}
\label{sequence3}
0 \rightarrow f^*\omega_C \rightarrow \Omega_S \rightarrow \Omega_{S/C} \rightarrow 0.
\end{equation}
Dualizing \eqref{sequence3} we obtain:
\[
0 \rightarrow \Omega_{S/C}^{\vee}  \rightarrow T_S \rightarrow f^*T_{C} \rightarrow  \mathcal{E}:=\mathcal{E}xt(\Omega_{S/C},\mathcal{O}_C) \rightarrow 0,
\]
where $\Omega_{S/C}^{\vee}:=\mathcal{H}om(\Omega_{S/C}, \mathcal{O}_{C})$. Here the sheaf $\Omega_{S/C}^{\vee}$ is locally free (see for example \cite{Se}). Denote by $J$ the image of  $T_S \rightarrow f^*T_{C}$. Then we get the following commutative diagram:
$$
\xymatrix{
0 \ar[r] & \Omega_{S/C}^{\vee} \ar[r]  \ar[d]  & T_{S} \ar[r]  \ar[d]  & J  \ar[r]   \ar[d]  & 0 & \\
0 \ar[r] & \Omega_{S/C}^{\vee}  \ar[r] & T_{S} \ar[r] & J^{\vee \vee} \ar[r] & Q \ar[r]  &0 \\ }
$$
where  the cokernel  $Q \simeq J^{\vee \vee}/J$ is  is the structure sheaf of a zero dimensional subscheme $Z$. In particular we get: 
\[
0 \rightarrow  \Omega_{S/C}^{\vee}  \rightarrow T_{S}  \rightarrow I_{Z} \otimes J^{\vee \vee}  \rightarrow 0.
\]
From \cite[Proposition 3.1]{Se},  we know that $J^{\vee \vee} \simeq f^*\omega_{C}^{\vee}(-D)$ and $ \Omega_{S/C}^{\vee } \simeq  \omega_{S/C}^{\vee}(D)$. Hence we deduce
\[
0 \rightarrow  f^*\omega_{C} (D)  \rightarrow T_{S}  \otimes \omega_{S} \rightarrow I_{Z} \otimes   \omega_{S/C}(-D)  \rightarrow 0.
\]
Since for a smooth surface: $ \Omega_{S} \simeq T_{S} \otimes \omega_{S}$, we deduce \eqref{shortexactsequence} as desidered. 
\\
\\
As in \cite{HP} we consider the surface
\[
 S_1:=\mathbb{P}(I_Z \otimes \omega_{S / \mathbb{P}^1}(-D)).
 \]
 Denote by $2F$ the class of the fiber. From  \eqref{shortexactsequence} we see that $S_1$ is a divisor in $\pi: \mathbb{P}(\Omega_S) \rightarrow S$, and 
 \begin{equation}
 \label{hpdivisor}
S_1  \in |L-\pi^*(c_1(f^*\omega_{\mathbb{P}^1}(D)))|=|L+4\pi^*F-\pi^*D|.
 \end{equation}
Moreover, since $Z$ is a local complete intersection of codimension $2$,  we have that $S_1$ coincides with the blow-up of $S$ at $Z$. In particular it is an integral surface. We can actually say more in the case of a general Enriques surface.
\begin{lemma}
\label{singularity}
Suppose that the reduction of each singular fiber of $f: S \rightarrow \mathbb{P}^1$ is an  irreducible nodal curve. Then $S_1$ is a smooth irreducible surface. In particular this holds for any elliptic fibration on a very general Enriques surface.
\end{lemma}
\begin{proof}
We only need to show that it is smooth. Let $ x$ be a point in the support of $Z$. By \cite[Proposition 3.1]{Se}  we know that $\mathcal{O}_{Z}$ is supported at the singular points of the reduction of the fibers. More precisely we have:
\begin{equation}
\operatorname{length}(\mathcal{O}_{Z,x})=\sum_{k \in L} (2\delta_{x}(E_{k})-\nu_{x}(E_{k}))+ \sum_{(k,l) \in L \times L, \ k \neq l}(E_{k} \cdot (E_{l})_{x}+ 1,
\end{equation}
where 
\begin{itemize}
\item $ \{E_k\}_{k \in L}$ is the set of components of $f^{-1}(f(x))$ passing through $x$. 
\item  $\delta_x(E):=\operatorname{codimension}$ of $\mathcal{O}_{E,x}$ in its normalization (as vector space over the ground field). 

\item $\nu_{x}(E)$:= number of analytic branches of $E$ at $x$. 

\item $(E_{k} \cdot E_{l})_{x}$:= local intersection number of $E_{k}$, $E_{l}$ at $x$.
\end{itemize}
Suppose $E:=f^{-1}(f(x))_{red}$ is an irreducible nodal curve, and $x \in Z$. Then the multiplicity of $Z$  in $x$ is $1$. Indeed the number of analytic branches is $2$, and the codimension of  $\mathcal{O}_{E,x}$ in its normalization (as vector space over the ground field) is one. In particular, if we suppose that the fibers of the fibration are at most nodal, we get that $S_1$ is just the blow-up of $c_{2}(\mathcal{O}_{Z})$ distinct points. For the claim about very general Enriques surfaces we refer to Remark \ref{remarknodalfibres}.
\end{proof}

We will also need the following lemma:
\begin{prop}
\label{globallygeneratedeNriques}
Let $S$ be an unnodal  Enriques surface. Let $F_1$ and $F_2$ be two half pencils such that $F_1 \cdot F_2=1$. Set $H:=F_1+F_2$. Then $\Omega_S(4H)$ is globally generated. 
\end{prop}
\begin{proof}
We use the same construction and notations as in Proposition \ref{doublecover}. Let $R$ be the ramification curve of the double covering (a $K3$ surface) $h:S' \rightarrow Q=\mathbb{P}^1 \times   \mathbb{P}^1  \subset \mathbb{P}^3$  branched along a curve $B \subset \mathbb{P}^1 \times \mathbb{P}^1$ of bidegree $(4,4)$.   Set $H':=f^*H=h^*\mathcal{O}_{\mathbb{P}^1 \times \mathbb{P}^1}(1,1)$.  Recall that $R \in  |2H'|$. We have the following short exact sequence:
\begin{equation}
\label{sequence1}
0 \rightarrow h^*\Omega_{Q}  \rightarrow  \Omega_{S'}  \rightarrow  \omega_h \simeq \mathcal{O}_{R}(-R) \rightarrow 0,
\end{equation}
which we tensor by $4H' \sim 2R$:
\begin{equation}
\label{sequence2}
0 \rightarrow h^*\Omega_{\mathbb{P}^1 \times \mathbb{P}^1}(4,4) \rightarrow  \Omega_{S'}(4H')  \rightarrow    \omega_R \rightarrow 0.
\end{equation}
Notice that since $f^*(4H)=4H'=h^*\mathcal{O}_{\mathbb{P}^1 \times \mathbb{P}^1}(4,4)$, and $f \circ \tau=f$,  we have: $\tau^*(4H')=\tau^*(f^*(4H))=f^*(4H)=4H'$. The global sections of $\Omega_{S}(4H)$ are canonically identified  with elements in $H^0(\Omega_{S'}(4H'))$ which are invariant under $\tau$.  Hence we need to show that $\tau$-invariant global sections generate $\Omega_{S'}(4H')$ at every point $p \in S'$.

We first show that the result follows from the following assumptions:
\begin{enumerate}
    \item $h^*\Omega_{\mathbb{P}^1\times\mathbb{P}^1}(4,4)$   is generated  by  $\tau$-invariant global sections at every point $ p \in S'$ ($\tau$-invariant sections are sections whose image in $\Omega_{S'}(4H')$ is $\tau$-invariant).

    \item  $\tau_{|R}$-invariant global sections of $H^0(R,\omega_R)$ generate the fiber of $\omega_R$ at every points $p \in R$   
\end{enumerate}
Let $p \in S$ be a point. Since $H^1(S',h^*\Omega_{\mathbb{P}^1\times\mathbb{P}^1}(4,4))=0$ we have the following diagram, where vertical arrows are evaluation morphisms:
\[
\resizebox{\linewidth}{!}{$
\begin{CD}
@. 0 @>>> H^0(S',h^*\Omega_{\mathbb{P}^1\times\mathbb{P}^1}(4,4))\otimes \mathcal{O}_{S'}
  @>>> H^0(S',\Omega_{S'}(4H'))\otimes \mathcal{O}_{S'}
  @>>> H^0(R,\omega_R)\otimes \mathcal{O}_{S'} @>>> 0 \\
@. @VVV @VVV @VVV @VVV @. \\
0 @>>> \operatorname{Tor}_1^{\mathcal{O}_{S'}}(\mathcal{O}_R,k(p))
  @>>> h^*\Omega_{\mathbb{P}^1\times\mathbb{P}^1}(4,4)\otimes k(p)
  @>>> \Omega_{S'}(4H')\otimes k(p)
  @>>> \omega_R \otimes k(p) @>>> 0.
\end{CD}
$}
\]

First suppose that $p \notin R$. Then, since $h^*\Omega_{\mathbb{P}^1\times\mathbb{P}^1}(4,4)$ is globally generated by $\tau$-invariant sections, it immediately follows that also $\Omega_S(4H')$ is. Now take $p \in R$. Let $s_1 \in H^0(S',\Omega_S(4H')$ be a $\tau$-invariant global section coming from   $h^*\Omega_{\mathbb{P}^1\times\mathbb{P}^1}(4,4)$.   Let  $s_2$ be a lift of a $\tau_{|R}$-invariant section in $H^0(R,\omega_R)$ which does not vanish in $p$.  From the diagram we see that $s_1(p)$ generates a one dimensional subspace of $\Omega_{S'}(4H') \otimes k(p) \simeq k^{\oplus 2}$, and hence $\langle s_1(p), s_2(p)\rangle$ generates the fiber of $\Omega_{S'}(4H')$ at p and the conclusion follows. Then we just need to show that $(1)$ and $(2)$ hold.

We start with $(1)$. We need a more careful analysis of the geometry of the Horikawa construction. Recall that the double cover $S' \rightarrow \mathbb{P}^1 \times \mathbb{P}^1$ is determined by a homogenous polynomial  $s$ of  bidegree $(4,4)$ in the variables $x_0,x_1,y_0,y_1$, which is invariant with respect to $i:(x_0,x_1,y_0,y_1) \rightarrow (x_0,-x_1,y_0,-y_1)$, and such that $V(f)$  is smooth and does not contain the $4$ fixed points: $(0,0), (0,\infty), (\infty,0), (\infty,\infty)$, of the involution $i$. Since $S'$ is a branched double cover, it can be locally described in this way:  let $t \in H^0(\mathbb{P}^1 \times  \mathbb{P}^1, L)$ be a local generator of  $L:=\mathcal{O}_{\mathbb{P}^1 \times \mathbb{P}^1}(2,2)$, and let $f \in k[u,v]$ be a local function such that $s=ft^{\otimes 2}$. Then $S'$ is  given (in the chart $U_{00}:=\{x_0 \neq 0\} \cap \{y_0 \neq 0\})$ by:
\[
\operatorname{spec}(k[u,v,t]/(t^2-f)) \rightarrow  \operatorname{spec}(k[u,v]).
\]
Observe that $h \circ \tau=i \circ h$ (recall that the involution $\tau$ on $S'$ is induced by $i$). Since $i^*u=-u$, and $i^*v=-v$, $h^*u=u$, and $h^*v=v$,  we find $\tau^*(u)=-u$, $\tau^*(v)=-v$, and $\tau^*f(u,v)=f(u,v)$, since the $s$ we started with is invariant under $i^*$.  We also have $\tau^*t=-t$. Indeed since $k[u,v,t]/(t^2-f) \simeq  k[u,v] \oplus k[u,v]t$  as $k[u,v]$-algebras (where the latter has the ring structure induced by $t^2=f)$,  $\tau^*t$ can be locally  written in a unique way as  $a+bt$, where $a,b \in  k[u,v]$. Then 
\[
(\tau^*t)^2=(a+bt)^2=\tau^*f=f
\]
in $k[u,v,t]/t^2-f)$. Since $t^2=f$, this gives $(a +bt)^2=t^2$ which implies  $a=0$ and  $b=+1$ or $b=-1$. Then the  (Enriques) involution $\tau$ satisfy $\tau^*(u,v,t)=(-u,-v,bt)$ in the ring $k[u,v,t]/(t^2-f)$. If we had $b=1$, then the image of the point $(u-0,v-0,t-\sqrt{f(0,0)})$, would be $(-u-0,-v-0,t-\sqrt{f(0,0)})=(u-0,v-0,t-\sqrt{f(0,0)})$. Then this would be a fixed point of the involution $\tau$: the choice $b=1$ is not compatible with the the fact that $\tau$ has no fixed points, hence $b=-1$,  and the action of $\tau$ on function is locally given by:
\[
\tau^*(u,v,t)=(-u,-v,-t)
\]
The induced action on $\Omega_{S'}$ is then given by $du \rightarrow -du$, $dv \rightarrow -dv$,  $dt \rightarrow -dt$.  Since 
\[
h: \operatorname{spec} (k[u,v] \oplus k[u,v][t]) \simeq \operatorname{spec}(k[u,v,t]/(t^2-f)) \rightarrow  \operatorname{spec}(k[u,v])
\]
in this open charts we have the decomposition
\begin{equation}
\label{decompo}
(h^*\Omega_{\mathbb{P}^1\times\mathbb{P}^1}(4,4)) \simeq \Omega_{\mathbb{P}^1\times\mathbb{P}^1}(4,4)_{|U_{00}} \oplus t \Omega_{\mathbb{P}^1\times\mathbb{P}^1}(4,4)_{|U_{00}},
\end{equation} 
and every section can be written locally as as $s_1 +s_2t$ where $s_1,s_2 \in \Omega_{\mathbb{P}^1\times\mathbb{P}^1}(4,4)_{|U_{00}}$. With the identification \eqref{decompo}, the morphism $h^*\Omega \rightarrow \Omega_{S'}$ is just the inclusion. Recall that we aim to show that $\tau$- invariant global sections of $h^*\Omega_{\mathbb{P}^1\times\mathbb{P}^1}(4,4)$ generates the fiber at each point. The decoposition \eqref{decompo} is globally expressed as:
\small{
\[
H^0(S',(h^*\Omega_{\mathbb{P}^1\times\mathbb{P}^1}(4,4)))  \simeq H^0(h_*(h^*\Omega_{\mathbb{P}^1\times\mathbb{P}^1}(4,4))) \simeq  H^0(\Omega_{\mathbb{P}^1\times\mathbb{P}^1}(4,4))) \oplus H^0(\Omega_{\mathbb{P}^1\times\mathbb{P}^1}(4,4)) \otimes_{\mathcal{O}_{\mathbb{P}^1 \times \mathbb{P}^1}} L^{-1}).
\]
}
We first focus on $\tau$-invariant global sections coming from $\Omega_{\mathbb{P}^1 \times \mathbb{P}^1}(4,4)$. In this chart (i.e.  $u=x_1/x_0$, $v=y_1/y_0$), these are locally given by 
\[
u^iv^jx_0^4y_0^4du, \ \ u^{i'}v^{j'}x_0^4y_0^4dv,
\]
where   $ 0 \leq i \leq 2$,  $ 0 \leq j \leq 4$,  $i+j$ odd, and  $ 0 \leq i' \leq 4$,  $ 0 \leq j' \leq 2$  with $i'+j'$ odd. It is easy to see that corresponding global sections generates the vector bundle away from the four fixed points $\{(0,0),(0,\infty),(\infty,0),(\infty,\infty)\}$. To conclude observe that there are four $\tau$-invariant global sections coming from the $H^0(\Omega_{\mathbb{P}^1\times\mathbb{P}^1}(4,4)) \otimes_{\mathcal{O}_{\mathbb{P}^1 \times \mathbb{P}^1}} L^{-1})$-piece which in this chart write as:
\[
tx_0^4y_0^4du \ \text{and}  \ v^2tx_0^4y_0^4du,
\]
and
\[
tx_0^4y_0^4dv \ \text{and}  \ u^2tx_0^4y_0^4dv.
\]
It is easy to see that that for each point in  $\{(0,0),(0,\infty),(\infty,0),(\infty,\infty)\}$ there are combination of these sections which generate the fiber over each point.
\\
For part $(ii)$, this easily follows from the fact that $R \subset S$ is a $2:1$ (étale) cover of a smooth curve in $S$ (see Proposition \ref{doublecover})) with the involution given by $\tau_{|R}$. Hence the $\tau$-invariant part of $H^0(R,\omega_R)$ corresponds to $H^0(f(R),\omega_{f(R)})$. This is globally generated since  $g(f(R))=5$. 
\end{proof}

\begin{remark}
It is easy to see that the same argument does not work with $\Omega_S(2H)$.
\end{remark}

We will also need this basic lemma about intersection of divisors in $\mathbb{P}(\Omega_S)$.
\begin{lemma}
\label{lemmaintersection}
Let $S$ be an Enriques surface, and $\pi: \mathbb{P}(\Omega_S) \rightarrow S$ the bundle morphism. Then for any $H, H' \in \operatorname{Pic}(S)$ we have
\[
\begin{aligned}
&L^3 = s_2(\Omega_S)=-12, \ \  &L^2\cdot \pi^*H = s_1(\Omega_S)\cdot H=0,\\
&L\cdot (\pi^*H)(\pi^*H') = H\cdot H', \ \ &(\pi^*H)^i(\pi^*H')^j = 0 \ \text{for all} \ i+j=3.
\end{aligned}
\]
\end{lemma}
\begin{proof}
It follows easily from the definition of Segre classes.
\end{proof}
Now we are able to construct the destabilizing family.
 \begin{teo}
 \label{maintheorem2}
Let $S$ be a very  general Enriques surface. Then there exists a line bundle $H'$ with $H'^2=48$ and $\phi(H')=4$, and a positive dimensional family of smooth irreducible curves in $|H'|$ such that $\Omega_{S}|_C$ is not semistable.
\end{teo}
\begin{proof}
Since $S$ is a very general Enriques surface we can assume that there are two half pencils $F_1$ and $F_2$ such that $F_1 \cdot F_2=1$ (again \cite[Theorem 17.7]{BHPV}). Moreover, by Remark \ref{remarknodalfibres}, we can assume that the elliptic fibration given by $|F_1|$ has exactly two double fibres (and these are precisely the multiple fibres), the underlying reduced fibres are smooth, and there are exactly $12$ singular (non reduced) fibres, all of which are nodal integral curves. We set $H:=2F_1+2F_2$. In this case  \eqref{hpdivisor} becomes:
\begin{equation}
S_1 \in |L+2\pi^*F_1+\pi^*K_S|
\end{equation}
since $D=F_1+F_1'$,  and $F_1' \sim F_1+K_S$. Here $F_1'$ denotes the other half fiber, as usual. By the previous discussion and   Lemma \ref{singularity}, we know that $S_1$ is a smooth irreducible surface obtained as the blow-up of $S$ in $12$ points. 

From Proposition \ref{globallygeneratedeNriques} we know that  $\Omega_{S}(2H)$ is globally generated (notice that $2H$ here is $4H$ in the notation of the Proposition).  Then,  pulling-back,  we obtain that   $\pi^*(\Omega_S(2H))$ and the (twisted) tautological quotient:  $L+2\pi^*H$ are globally generated on $\mathbb{P}(\Omega_S)$. Set $l:=L|_{S_1}$, and $h:=H|_{S_1}$. Then, the linear system $|l+2h|$ is globally generated on $S_1$.  Moreover, $(l+2h)^2=(L+2\pi^*F_1+\pi^*K_S)(L+2\pi^*H)^2=4H^2+8F_1H-12=32+16-12>0$. Hence, by Bertini, a  general element in $|l+2h|$ is a smooth irreducible curve which we denote by $\tilde{C}$. Now we show that $\operatorname{deg}(L_{|\tilde{C}}) <0$.  We have:
\[
\begin{aligned}
\operatorname{deg}(l|_{\tilde{C}})
&:= l(l+2h) \\
&= L(L+2\pi^*F_1+\pi^*K_S)(L+4\pi^*F_1+4\pi^*F_2) \\
&= -12 + 8F_1F_2=-4<0.
\end{aligned}
\]
 Now denote by  $g$ the restriction of the bundle morphism $\mathbb{P}(\Omega_S) \rightarrow S$ to $S_1$.  We are going to show that $g_{|\tilde C}$ is an isomorphism onto its image following the same approach as in \cite{GO}. Recall that $S_1$ is the blow-up in $12$ points. Let $E=\sum_{i=1}^{12} E_i$ be the exceptional divisor of $g:S_1 \rightarrow S$. Since $K_{\mathbb{P}(\Omega_S)}=-2L$, by adjunction $K_{S_1}=(2\pi^*F_1+\pi^*K_S-L)|_{S_1}$. On the other hand we have $K_{S_1}=g^*K_S+E$. Then:
 \[
 \tilde C \cdot E=\tilde C \cdot (\pi^*K_{S} +E)=\tilde C \cdot K_{S_1}=(L+2\pi^*F_1+\pi^*K_S)(L+2\pi^*H)(2\pi^*F_1 +\pi^*K_S-L)=12.
 \]
 Again ,we have used  Lemma  \ref{lemmaintersection}.  Actually we can say that $\tilde{C} \cdot E_i=1$ for each $i$:  since $E_i$ is a smooth rational curve, and $S$ is a very general Enriques surface, $E_i$ must be contracted by $g$. Hence it  is a fiber of the bundle $\mathbb{P}(\Omega_S)$. Since  $1=E_i\cdot L$ for every $i$, we have $E_i \cdot (l+2h)=1$. Then a general $\tilde{C} \in |l+2h|$ is a smooth irreducible curve which intersects each exceptional divisor in $1$ point, and $C:=g(\tilde{C}) $ is  a smooth irreducible curve. Using that $g^*C-E=\tilde C$ we find $
C^2=\tilde C^{\,2}+12=36+12=48$. Now observe that $F_1C=g^*F_1 g^*C=g^*F_1 \tilde C=\pi^*F_1(L+2\pi^*F_1+\pi^*K_S)(L+2\pi^*H)=4$. It is immediate to see that if $F'$ is another half pencil, we have $F'C \geq 4$. This gives $\phi(C)=4$.  
\end{proof}

\begin{remark}
\label{remarknodalfibres}
If $S$ is a very general Enriques surface, then $S$ contains no smooth rational curves. Moreover, every genus-one fibration on $S$ has exactly $12$ reduced singular fibres, all of which are nodal cubics, and for each genus-one fibration on $S$, both half-fibres are irreducible, smooth elliptic curves (see \cite{ramschutt}, \cite[Remark 5.6]{martin})

\end{remark}

\begin{remark}
\label{remarkcomparison}
The idea (coming from \cite{GO}) of constructing destabilizing curves on Enriques surfaces as complete intersections inside $\mathbb{P}(\Omega_S)$ of two smooth surfaces required some careful work in order to be carried out.
Let $H_1$ and $H_2$ be two line bundles on $S$, and assume there exist $S_1 \in |L+\pi^*H_1|$ and $S_2 \in |L+\pi^*H_2|$ whose intersection is a smooth irreducible curve $\tilde{C}$, and such that $\deg(L_{|\tilde{C}})<0$.
Then the same construction as in the Theorem shows that the image of $\tilde{C}$ in $S$ is a smooth irreducible curve which destabilizes the cotangent bundle.
The difficulty lies, of course, in the assumption.
To ensure that there exist divisors $S_i \in |L+\pi^*H_i|$, one either gives an ad-hoc construction (as we did in the theorem,  using Lemma~\ref{singularity}), or shows that $|L+\pi^*H_i|$ is globally generated.  
One is then usually led to analyze the global generation of $\Omega_S(H)$.
When $H$ is not a multiple of a polarization this is usually a hard problem, so one considers $H_i=m_iH'_i$ with $m_i\ge 2$.
Requiring that $L\cdot(L+m_1\pi^*H'_1)\cdot(L+m_2\pi^*H'_2)<0$ gives the constraint $H'_1\cdot H'_2<3$. Moreover by Hodge index theorem
\[
\phi(H'_1)\phi(H'_2)\le \sqrt{(H'_1)^2}\,\sqrt{(H'_2)^2}\le H'_1\cdot H'_2<3.
\]
Thus we get a very limited choice for the $H'_i$: either both $\phi(H_i')=1$, or  $\phi(H'_1)=1$ and $\phi(H'_2)=2$, or the reverse situation holds.
In particular, at least one of the $H'_i$ is not globally generated, and the other is  not very ample (recall Theorem \ref{etoegkveraglob}). 
In this sense if one wants to follow the approach in \cite{GO} the construction is quite optimal.

\end{remark}


\bibliography{references}

\bibliographystyle{alpha}

\end{document}